\documentclass[12pt,a4paper]{amsart}
\usepackage{url}
\usepackage{amsmath,amsthm,amsfonts,amssymb,latexsym}
\usepackage{todonotes}

\newtheorem{theorem}{Theorem}[section]

\newtheorem{lemma}[theorem]{Lemma}
\newtheorem{claim}[theorem]{Claim}

\theoremstyle{definition}

\newcommand{\w}{\omega}

\newcommand{\IP}{\mathbb P}

\newcommand{\bigvid}{\hat{\ \ }}

\newcommand{\uhr}{\upharpoonright}

\newcommand{\la}{\langle}
\newcommand{\ra}{\rangle}

\newcommand{\hot}{\mathfrak}

\newcommand{\nothing}[1]{}

\title[A countably compact group with the non-countably pracompact square]
{A countably compact topological group with the non-countably pracompact square}

\author{Serhii Bardyla, Alex Ravsky, and Lyubomyr Zdomskyy}

\address{Universit\"at Wien,
Institut f\"ur Mathematik, Kurt G\"odel Research Center
 Augasse 2-6, UZA 1 - Building 2, 1090 Wien, Austria.}
\email{sbardyla@yahoo.com}
\urladdr{http://www.logic.univie.ac.at/$\tilde{\ }$bardylas55/}

\address{Department of Analysis, Geometry and Topology,
Pidstryhach Institute for Applied Problems of Mechanics and Mathematics,
National Academy of Sciences of Ukraine, Naukova 3-b, Lviv, 79060, Ukraine.}
\email{alexander.ravsky@uni-wuerzburg.de}

\address{Universit\"at Wien,
Institut f\"ur Mathematik, Kurt G\"odel Research Center,
 Augasse 2-6, UZA 1 - Building 2, 1090 Wien, Austria.}
\email{lzdomsky@gmail.com}
\urladdr{http://www.logic.univie.ac.at/$\tilde{\ }$lzdomsky/}


\thanks{The first and third authors were partially supported by
the Austrian Science Fund FWF Grant I 3709-N35.  The third author also thanks
the Austrian Science Fund FWF (Grant I 2374-N35) for generous support for this research.\\
The main part of it was made during the visit of the second
author at the Institute for Mathematics (KGRC) at the University of Vienna in December 2019,
supported by the above FWF grants. The second author thanks KGRC for the hospitality.}

\begin{document}
\begin{abstract}
Under Martin's Axiom we construct a Boolean countably compact
topolo\-gi\-cal group whose square is not countably pracompact.
\end{abstract}

\maketitle

\section{Introduction}

Let us recall the definitions of compact-like spaces which will appear in this paper.

A space $X$ is called
\begin{itemize}
\item {\it countably compact} if each countable open cover of $X$ has a finite subcover;
\item {\it countably pracompact} if there exists a dense subset $A$ of $X$ such that each infinite subset $B\subset A$ has an accumulation point in $X$;
\item {\it feebly compact} if each locally finite family of nonempty open subsets of $X$ is finite;
\item {\it pseudocompact} if $X$ is Tychonoff and each continuous real-valued function on $X$ is bounded;
\end{itemize}

These notions relate as follows: 
$$\hbox{countably compactness}\Rightarrow\hbox{countably pracompactness}\Rightarrow\hbox{feebly compactness.}$$
Also, for Tychonoff spaces pseudocompactness is equivalent to feebly compactness.

By Tychonoff's theorem, the Tychonoff product of any family of compact spaces is compact. The
productivity of other compact-like properties can be a non-trivial problem. For instance, see ~\cite{Vau},~\cite{Ste},
 \cite{DPSW}, \cite{ScaSto}, \cite{BarRav}, and~\cite{GutRav}. In~\cite{Tka2},~\cite{Tka3}, and~\cite{Tka6} Tkachenko considered the
productivity of various properties of topological groups. 

Comfort and Ross in~\cite{ComRos} proved that the product of any family of
pseudocompact topological groups is pseudocompact.
On the other hand, Nov\'ak~\cite{Nov}  and  Teresaka~\cite{Ter}  constructed examples of two
countably  compact spaces whose  product is not pseudocompact.
This motivated Comfort to ask in a letter to Ross in 1966, whether there are countably compact topological groups
whose product is not countably compact.
This question is considered to be central in the theory of topological groups.
The first consistent positive answer was given by van Douwen~\cite{vDa80} under MA,
followed by Hart-van Mill~\cite{HarvMil} under MA$_{ctble}$.
The question has been studied extensively in recent decades, see~\cite{G-FTW, MT, ST}.
Finally, Hru\v s\'ak, van Mill, Ramos-Garc\'ia, and Shelah obtained a positive answer in ZFC,
see~\cite{HvMR-GS} for ``a presentable draft of the proof'' and more history and references.

A related question on productivity of countable pracompact paratopological groups was posed by the second author
in 2010 in the first version of the paper~\cite{BanRav}.
In the present paper we give a negative answer to this question,
constructing under MA a Boolean countably compact topological group whose square
is not countably pracompact.

\begin{theorem}\label{main}
(MA). There exists a countably compact subgroup $G$ of $2^{\hot c}$
without non-trivial convergent sequences such that $G^2$ is not countably
pracompact. 
\end{theorem}

For a subset $A$ of $2^\alpha$, where $\alpha$ is an ordinal, we
shall denote by $[A]$ the subgroup of $2^\alpha$ generated by $A$.

\section{Proof of Theorem~\ref{main}}

In our construction we ``add an $\varepsilon$'' to that in
\cite[\S~4]{vDa80}. In particular, we essentially use the notation
from \cite{vDa80}.

Let $\la I_\alpha:\w\leq\alpha<\hot c\ra$ be an enumeration of
$[\hot c]^\w$ such that $I_\alpha\subset\alpha$ for all $\alpha<\hot
c$. Let also
$$\{\big\la\la  i_{\alpha,n},j_{\alpha,n} \ra:n\in\w\big\ra\: : \: \alpha<\hot c\}$$
be an enumeration of all sequences $\la\la i_{n},j_{n} \ra:n\in\w
\ra$ of pairs of ordinals below $\hot c$, such that each such a
sequence appears in the enumeration cofinally often.

Put $\sigma_\w=\mu_\w=\w$ and by recursion on $\omega\le \alpha<\hot c$ we shall construct
increasing sequences $\la \sigma_\alpha\ra$ and $\la\mu_\alpha \ra$ of ordinals below $\hot c$
and a subgroup $E_\alpha=\{x_{\alpha,\xi}:\xi<\sigma_\alpha\}$ of $2^{\mu_\alpha}$
such that
\begin{itemize}
\item[$(i)$] $x_{\alpha,\xi}\neq x_{\alpha,\eta}$ for any distinct $\xi,\eta<\sigma_\alpha$;
\item[$(ii)$] $x_{\beta,\xi}\uhr\mu_\alpha= x_{\alpha,\xi}$ for each $\beta>\alpha$ and $\xi<\sigma_\alpha$;
\item[$(iii)$] If $\xi,\eta,\zeta<\sigma_\alpha$ and
$x_{\alpha,\xi}+x_{\alpha,\eta}=x_{\alpha,\zeta}$, then
$x_{\beta,\xi}+x_{\beta,\eta}=x_{\beta,\zeta}$ for each $\beta>\alpha$.
\end{itemize}
We shall also recursively construct a sequence
$\la\lambda_\gamma:\gamma<\hot c\ra$ of ordinals such that
$\lambda_\gamma<\sigma_{\gamma+1}$ for all $\gamma$, so that $G=E_{\hot c}$
will be countable compact because of
\begin{itemize}
\item[$(iv)$] $x_{\alpha,\lambda_\xi}$ is a cluster point of
$\{x_{\alpha,\eta}:\eta\in I_\xi\}$ for all $\xi<\alpha$.
\end{itemize}
Let $\{\la u_\rho,v_\rho \ra:\rho<|\sigma_\alpha|\}$ be a bijective
enumeration of $\sigma_\alpha\times\sigma_\alpha$.
 The failure of the countable pracompatness of $G\times G$ will follow from the following condition:
\begin{itemize}
\item[$(v)$] Suppose that
$i_{\alpha,n},j_{\alpha,n}<\sigma_\alpha$ for all $n\in\w$ and
 the family $\{x_{\alpha,i_{\alpha,n}}, x_{\alpha,j_{\alpha,n}}:n\in\w\}$
 is linearly independent\footnote{ $2^{\mu_\alpha}$ is a linear space over
 the field $F_2=\{0,1\}$.}. Then
for every $\rho<|\sigma_\alpha|$ there exists
$\epsilon\in\sigma_{\alpha+1}\setminus\sigma_\alpha$ such that
$x_{\alpha+1,u_\rho}(\epsilon)=x_{\alpha+1,v_\rho}(\epsilon)=0$ and
$$ \max\big\{x_{\alpha+1,i_{\alpha,n}}(\epsilon),
x_{\alpha+1,j_{\alpha,n}}(\epsilon)\big\}=1\mbox{ for all but
finitely many }n\in\w.$$
\end{itemize}
The key fact allowing us for the recursive construction of  the
$E_{\alpha}$'s is the following
\begin{lemma}\label{main_l}
(MA). Suppose that $\w\leq\mu<\hot c$,  $E$ is a subgroup of $2^\mu$
of cardinality $<\hot c$,   $\{x_n,y_n:n\in\w\}$ is a linearly
independent subset of $E$, and $x',y'\in E$. Let also $\mathcal
K\subset [E]^\w$ be of size $<\hot c$. Then there exists a
homomorphism $\phi:E\to \{0,1\}$ with the following properties:
\begin{enumerate}
\item $\phi(x')=\phi(y')=0$;
\item $\forall K\in\mathcal K\: \exists x\in K\: (\phi(x)=0)$;
\item $\forall K\in\mathcal K\: \exists x\in K\: (\phi(x)=1)$;
\item $\forall^* n\in\w \: (\max\{\phi(x_n),\phi(y_n)\} =1)$.
\end{enumerate}
\end{lemma}
\begin{proof}
Removing finitely many elements from the sequences $\la
x_n:n\in\w\ra$ and $\la y_n:n\in\w\ra$, we may assume that
$[\{x',y'\}]\cap[\{x_n,y_n:n\in\w\}]=\{0\}$. Under this assumption
we shall construct $\phi$ satisfying $(1)$-$(3)$ as well as the
following stronger version of $(4)$:
$$\forall n\in\w\:
(\max\{\phi(x_n),\phi(y_n)\} =1).$$ Let $\IP$ be the poset
consisting of  functions $p$ such that
\begin{itemize}
\item[$(A)$] $\mathrm{dom}(p)$ is a finite subgroup of $E$ and $p$ is a homomorphism from
$\mathrm{dom}(p)$ to $\{0,1\}$;
\item[$(B)$] $x',y'\in \mathrm{dom}(p)$ and $p(x')=p(y')=0$;
\item[$(C)$] $x_n\in \mathrm{dom}(p)$ iff $y_n\in \mathrm{dom}(p)$ for all
$n\in\w$;
\item[$(D)$] $\forall n\in\w\:
(x_n\in \mathrm{dom}(p)\Rightarrow\max\{p(x_n),p(y_n)\} =1)$
\item[$(E)$] $ [\{x_n,y_n: x_n\not\in \mathrm{dom}(p) \}]\cap
\mathrm{dom}(p)=\{0\}$.
\end{itemize}
\begin{claim}\label{ccc}
$\IP$ is ccc.
\end{claim}
\begin{proof}
Given any  $X\in [\IP]^{\w_1}$ and using the standard
$\Delta$-System argument we can find $Y\in [X]^{\w_1}$ such that
there exists a homomorphism $q$ from a finite subgroup of $E$ to $\{0,1\}$
with $p\cap p'=q$ for all
 $p\neq p'$ in $Y$. Let us fix any $p\in Y$ and note that
 $$E_0:=\mathrm{dom}(p)+
 [\{x_n,y_n: x_n\not\in \mathrm{dom}(p) \}]$$
is a countable subset of $E$. Therefore there exists $p'\in Y$ such
that $ E_0\cap
(\mathrm{dom}(p')\setminus\mathrm{dom}(q))=\emptyset$. Thus
\begin{eqnarray*}
\big((\mathrm{dom}(p')\setminus\mathrm{dom}(q))+\mathrm{dom}(p)\big)\cap
\big[\{x_n,y_n: x_n\not\in \mathrm{dom}(p) \}\big]=\emptyset,
\end{eqnarray*}
and hence also
\begin{eqnarray*}
\big(\mathrm{dom}(p')+\mathrm{dom}(p)\big)\bigcap \big[\{x_n,y_n:
x_n\not\in \mathrm{dom}(p')\cup \mathrm{dom}(p) \}\big]\subset\\
\subset \big(\mathrm{dom}(p')+\mathrm{dom}(p)\big)\bigcap
\big[\{x_n,y_n: x_n\not\in  \mathrm{dom}(p) \}\big] =\\
=
\big((\mathrm{dom}(p')\setminus\mathrm{dom}(q))+\mathrm{dom}(p)\big)\bigcap
\big[\{x_n,y_n: x_n\not\in \mathrm{dom}(p) \}\big] \bigcup \\
\bigcup \big(\mathrm{dom}(q)+\mathrm{dom}(p)\big)\bigcap
\big[\{x_n,y_n: x_n\not\in \mathrm{dom}(p) \}\big]=\\
=\emptyset\cup\{0\}=\{0\},
\end{eqnarray*}
which means that the unique homomorphism
$r:\mathrm{dom}(p')+\mathrm{dom}(p)\to\{0,1\}$ such that
$r\uhr\mathrm{dom}(p)=p$ and $r\uhr\mathrm{dom}(p')=p'$, satisfies
$(C)$, $(D)$, and $(E)$. It is obvious that $r$ also satisfies $(A)$
and $(B)$ and hence $r\in\IP$. Thus $r$ is a common extension of
$p,p'$ in $\IP$.
\end{proof}
\begin{claim}\label{dense1}
For every $p\in\IP$, $i\in 2$,
 $x\in E\setminus\mathrm{dom}(p)$ and $n\in\w$ such that
 $\{x_n,y_n\}\not\subset\mathrm{dom}(p)$ there exists
 $\IP\ni q\leq p$ such that $\{x,x_n,y_n\}\subset\mathrm{dom}(q)$
 and
 $q(x)=i$.
\end{claim}
\begin{proof}
If $ x\not\in [\mathrm{dom}(p)\cup\{x_k,y_k:k\in\w\}],$ then we can
first extend $p$ to $q_0$ with domain $[\mathrm{dom}(p)\cup\{x_n,y_n\}]$
such that $q_0(x_n)=q_0(y_n)=1$, and then further extend $q_0$ to
$q$ with domain $[\mathrm{dom}(q_0)\cup\{x\}]$ such that $q(x)=i$. So let
us assume that
$$ x=x_0+\sum_{j\in J}x_j+\sum_{l\in L}y_l $$
for some $x_0\in\mathrm{dom}(p)$ and finite $J,L\subset\w$ such that
at least one of them is non-empty (without loss of generality we
shall assume $J\neq\emptyset$) and $(\{x_j:j\in J\}\cup\{y_l:l\in
L\})\cap \mathrm{dom}(p)=\emptyset$. Let $q_0\leq p$ be the unique
element of $\IP$ with domain
$$[\mathrm{dom}(p)\cup\{x_n,y_n\}\cup\{x_m,y_m:m\in J\cup L\}]$$
such that $q_0\uhr(\{x_n,y_n\}\cup\{x_m,y_m:m\in J\cup L\})\equiv
1$. If $q_0(x)=i$ then $q_0$ is as required. Otherwise fix some
$j\in J$ and consider the unique element $\IP\ni q_1\leq p$  with
domain
$$[\mathrm{dom}(p)\cup\{x_n,y_n\}\cup\{x_m,y_m:m\in J\cup L\}]$$
such that $q_1(x_j)=0$ and
$$q_1\uhr\big((\{x_n,y_n\}\cup\{x_m,y_m:m\in J\cup
L\})\setminus\{x_j\}\big)\equiv 1.$$

It follows that $q_1$ is as required.
\end{proof}
Claim~\ref{dense1} immediately implies that for every $K\in\mathcal
K$,  $i\in 2$, and $x\in E$ the set
$$D_{K,i,x}=\{p\in\IP:x\in\mathrm{dom}(p)\:\wedge\:\exists x'\in K\cap \mathrm{dom}(p)\: (p(x')=i)\}$$
is dense in $\IP$. It is easy to see that if $G$ is a filter on
$\IP$ intersecting all the sets of the form $D_{K,i,x}$, then
$\phi:=\cup G$ satisfies the conditions $(1)$-$(4)$, which completes
the proof of Lemma~\ref{main_l}.
\end{proof}
\smallskip

Let us proceed now with the recursive construction of the subgroups
$E_\alpha$ fulfilling $(i)$-$(v)$. Suppose that we have already
constructed $E_\alpha$ for all $\alpha<\delta$ so that all the
relevant at this point instances of $(i)$-$(v)$ hold true. In the case
of limit $\delta$ set $\sigma_\delta=\sup_{\alpha<\delta}\sigma_\alpha$ and consider
any
$\xi<\sigma_\delta$. Let $\alpha<\delta$ be the minimal such that $\xi<\sigma_\alpha$.
Then we denote by $x_{\delta,\xi}$ the union $\bigcup_{\beta\in\delta\setminus\alpha}x_{\beta,\xi}$.
It is easy to check that the group $E_{\delta}=\{x_{\delta,\xi}: \xi<\sigma_{\delta}\}$ satisfies
conditions $(i)$-$(v)$.

So let us assume that $\delta=\alpha+1$ for some $\alpha$. If
$\max\{i_{\alpha,n},j_{\alpha,n}\}\geq\sigma_\alpha$ for some
$n\in\w$ or $\{x_{\alpha,i_{\alpha,n}},x_{\alpha,j_{\alpha,n}}:n\in\w\}$ is not
linearly independent (i.e., $(v)$ is vacuous at this stage of the
construction), then similarly as in the proof given in Case 2 on
\cite[p.~419-420]{vDa80} it can be showed that $E_{\delta}$ satisfies conditions $(i)$-$(v)$.
So let us assume that the premises in $(v)$ are satisfied and set
$\mu_{\alpha+1}=\mu_\alpha+|\sigma_\alpha|$. Recursively over
ordinals $\varepsilon\in(\mu_{\alpha+1}+1)\setminus\mu_\alpha$ we
shall first define $y_{\alpha,\xi,\varepsilon}\in 2^{\varepsilon}$
for all $\xi<\sigma_\alpha$ such that
\begin{itemize}
\item[$(i')$] $y_{\alpha,\xi,\mu_\alpha}\uhr\mu_\alpha = x_{\alpha,\xi}$
for all  $\xi<\sigma_\alpha$;
\item[$(ii')$] $y_{\alpha,\xi,\varepsilon_1}\uhr\varepsilon_0=
y_{\alpha,\xi,\varepsilon_0}$ for all
$\mu_\alpha\leq\varepsilon_0<\varepsilon_1\leq\mu_{\alpha+1}$ and
$\xi<\sigma_\alpha$;
\item[$(iii')$] If $\xi,\eta,\zeta<\sigma_\alpha$  and
$x_{\alpha,\xi}+x_{\alpha,\eta}=x_{\alpha,\zeta}$, then
$y_{\alpha,\xi,\varepsilon}+y_{\alpha,\eta,\varepsilon}=y_{\alpha,\zeta,\varepsilon}$
for all $\varepsilon\in (\mu_{\alpha+1}+1)\setminus \mu_\alpha$.
\item[$(iv')$] $y_{\alpha,\lambda_\xi,\varepsilon}$ is a cluster point of
$\{y_{\alpha,\eta,\varepsilon}:\eta\in I_\xi\}$ for all
$\xi<\alpha$.
\item[$(v')$] If $\varepsilon=\mu_\alpha+\rho$ for some $\rho<|\sigma_\alpha|$,
then
$y_{\alpha,u_\rho,\varepsilon+1}(\varepsilon)=y_{\alpha,v_\rho,\varepsilon+1}(\varepsilon)=0$
and
$$ \max\big\{y_{\alpha,i_{\alpha,n},\varepsilon+1}(\varepsilon),
y_{\alpha,j_{\alpha,n},\varepsilon+1}(\varepsilon)\big\}=1\mbox{ for
all but finitely many }n\in\w.$$
\end{itemize}
Suppose that the construction has been accomplished for all
$\varepsilon<\epsilon$. If $\epsilon$ is limit, then  letting
$y_{\alpha,\xi,\epsilon}=\bigcup_{\varepsilon<\epsilon}y_{\alpha,\xi,\varepsilon}$
for all $\xi<\sigma_\alpha$, we can easily check that $(i')$-$(v')$
are satisfied.
\smallskip

So let us assume that $\epsilon=\varepsilon+1$, where
$\varepsilon=\mu_\alpha+\rho$ with $\rho<|\sigma_\alpha|$.
Set $\mu=\varepsilon$,
$E=\{y_{\alpha,\xi,\varepsilon}:\xi<\sigma_\alpha\}$,
$x_n=y_{\alpha,i_{\alpha,n},\varepsilon}$ and
$y_n=y_{\alpha,j_{\alpha,n},\varepsilon}$ for all $n\in\w$,
$x'=y_{\alpha,u_\rho,\varepsilon}$,
$y'=y_{\alpha,v_\rho,\varepsilon}$, and
\begin{eqnarray*}
\mathcal K=\big\{\{y_{\alpha,\eta,\varepsilon}: \eta\in I_\xi\setminus\{\lambda_\xi\},
y_{\alpha,\eta,\varepsilon}\uhr
L=y_{\alpha,\lambda_\xi,\varepsilon}\uhr L \}\ :\  \xi<\alpha,
L\in[\varepsilon]^{<\w}\big\}.
\end{eqnarray*}
By Lemma~\ref{main_l} there exists a homomorphism $\phi:E\to \{0,1\}$
satisfying conditions $(1)$-$(4)$ in the formulation thereof.
For every $\xi<\sigma_\alpha$ we set
$$y_{\alpha,\xi,\varepsilon+1}=y_{\alpha,\xi,\varepsilon}
\bigvid \phi(y_{\alpha,\xi,\varepsilon})$$ and note that $(ii')$,
$(iii')$ and $(v')$ are satisfied by our construction of $\phi$.
(E.g., $(v')$ follows from conditions $(1)$ and $(4)$ in
Lemma~\ref{main_l} and our choice of $x',y',x_n,y_n$ made above,
where $n\in\w$.) To prove $(iv')$ for $\varepsilon+1=\epsilon$ let
us fix $\xi<\alpha$ and a finite subset $L'$ of $\varepsilon+1$. We
need to find $\eta\in I_\xi\setminus\{\lambda_\xi\}$ such that
\begin{equation}\label{zamykan}
y_{\alpha,\lambda_\xi,\varepsilon+1}\uhr L'=
y_{\alpha,\eta,\varepsilon+1}\uhr L'.
\end{equation}
If $\varepsilon\not\in L'$ then such $\eta\in I_\xi\setminus\{\lambda_\xi\}$
exists by $(iv')$ for $\varepsilon$, which holds by our
assumption. So suppose that $\varepsilon \in L'$ and put
$L=L'\cap\varepsilon$ and
$i=\phi(y_{\alpha,\lambda_\xi,\varepsilon})$ (so that
$y_{\alpha,\lambda_\xi,\varepsilon+1}$ becomes
$y_{\alpha,\lambda_\xi,\varepsilon}\bigvid i$). By conditions $(2)$
and $(3)$ in Lemma~\ref{main_l} there exists $x\in K$ with
$\phi(x)=i$, where
$$K:=\{y_{\alpha,\eta,\varepsilon}: \eta\in I_\xi\setminus\{\lambda_\xi\},
y_{\alpha,\eta,\varepsilon}\uhr
L=y_{\alpha,\lambda_\xi,\varepsilon}\uhr L \}\in\mathcal K. $$
Thus $x=y_{\alpha,\eta,\varepsilon}$ for some $\eta\in
I_\xi\setminus\{\lambda_\xi\}$. It is easy to check now that $y_{\alpha,\eta,\varepsilon+1}$
satisfies (\ref{zamykan}), which completes our construction of
$y_{\alpha,\xi,\varepsilon}$ for all
$\varepsilon\in(\mu_{\alpha+1}+1)\setminus\mu_\alpha$ and
$\xi<\sigma_\alpha$ satisfying $(i')$-$(v')$.
\smallskip

Now set $x_{\alpha+1,\xi}:=y_{\alpha,\xi,\mu_{\alpha+1}}$ for all
$\xi<\sigma_\alpha$ and $E_{\alpha+1}^0:=\{x_{\alpha+1,\xi}:\xi<\sigma_\alpha\}$.
Furthermore set $\sigma_{\alpha+1}:=\sigma_\alpha+\sigma_\alpha$ and
fix $c\in 2^{\mu_{\alpha+1}}\setminus E_{\alpha+1}^0 $ so that if
$\overline{\{x_{\alpha+1,\eta}:\eta\in I_\alpha\}}\not\subset
E_{\alpha+1}^0$,
 then $c\in \overline{\{x_{\alpha+1,\eta}:\eta\in
I_\alpha\}}\setminus E_{\alpha+1}^0$ (here we consider the closure
in $2^{\mu_{\alpha+1}}$).

Enumerate $c+E_{\alpha+1}^0$ bijectively
as $\{x_{\alpha+1,\xi}:\sigma_\alpha\leq\xi<\sigma_{\alpha+1}\}$ and
set
$E_{\alpha+1}=\{x_{\alpha+1,\xi}:\xi<\sigma_{\alpha+1}\}=E_{\alpha+1}^0\cup(c+E_{\alpha+1}^0)$.
Finally, if $c\in \overline{\{x_{\alpha+1,\eta}:\eta\in
I_\alpha\}},$ then let $\lambda_\alpha$ be the number
$\xi\in\sigma_{\alpha+1}\setminus\sigma_\alpha$ so that
$c=x_{\alpha+1,\xi}$. Otherwise we have
$\overline{\{x_{\alpha+1,\eta}:\eta\in I_\alpha\}}\subset
E_{\alpha+1}^0$, and hence there exists
$\lambda_\alpha<\sigma_\alpha$ such that
$x_{\alpha+1,\lambda_\alpha}$ is a cluster point of
$\{x_{\alpha+1,\eta}:\eta\in I_\alpha\}$. This completes our
construction of  $E_\alpha,$ $\alpha<\hot c$, satisfying conditions
$(i)$-$(v)$.
\smallskip

For every $\xi,\nu<\hot c$   find $\alpha<\hot c$ such that
$\xi<\sigma_\alpha$ and $\nu<\mu_\alpha$ and set
$x_\xi(\nu)=x_{\alpha,\xi}(\nu)$.
Condition $(ii)$ ensures that this
definition does not depend on the choice of $\alpha$. Let's check that
$G:=\{x_\xi:\xi<\hot c\}$ is as required in Theorem~\ref{main}.

The group $G$ is countably compact because by $(ii)$ and $(iv)$ we have that
$x_{\lambda_\xi}$ is a cluster point of $\{x_\eta:\eta\in I_\xi\}$
for all $\xi<\hot c$.

Now let $Z$ be any dense subset of $G^2$. It is easy to find
a countable subset $C:=\{\la x_{i_n},
x_{j_n}\ra:n\in\w\}$ of $Z$ such that the indexed set
$\{x_{i_n}, x_{j_n}:n\in\w\}$ is linearly independent.

We claim that $C$ is closed and discrete in $G^2$. Indeed, otherwise
$C$ has a cluster point $\la x_u,x_v\ra$ for some
$u,v<\hot c$. Since there are $\hot c$-many $\alpha$ such that
$i_n=i_{\alpha,n}$ and $j_n=j_{\alpha,n}$ for each $n\in\omega$,
we can find such $\alpha>\sup\{i_n,j_n:n\in\w\}$.
Moreover, we can choose $\alpha$ so large that $|\sigma_\alpha|>\max\{u,v\}$
and the indexed set $\{x_{\alpha,i_n}, x_{\alpha,j_n}:n\in\w\}$
is a linearly independent subset of $E_{\alpha}$. Pick
$\rho<|\sigma_\alpha|$ such that $\la u,v\ra=\la
u_\rho,v_\rho\ra$. By $(v)$ we can find
$\epsilon\in\sigma_{\alpha+1}\setminus\sigma_\alpha$ such that
$x_{\alpha+1,u_\rho}(\epsilon)=x_{\alpha+1,v_\rho}(\epsilon)=0$
(that is $x_u(\epsilon)=x_v(\epsilon)=0$) and
$$ \max\big\{x_{\alpha+1,i_{\alpha,n}}(\epsilon),
x_{\alpha+1,j_{\alpha,n}}(\epsilon)\big\}=1\mbox{ for all but
finitely many }n\in\w,$$ that is $ \max\{x_{i_{n}}(\epsilon),
x_{j_{n}}(\epsilon)\}=1$  for all these $n\in\w$. It follows that
the intersection of $C$ with
the open neighborhood
$$ \{\la x,x'\ra\in G^2 : x(\epsilon)=x'(\epsilon)=0 \}$$
of $\la x_u,x_v\ra$ is finite, a contradiction.

Finally, we claim that $G$ has only trivial convergent
sequences. Indeed, otherwise there exists an injective
sequence $\la x_{\xi_n}:n\in\w\ra\in G^\w$ converging to some
$x_\xi$. Passing to a subsequence, if necessary, we may assume that
$\{x_{\xi_n}:n\in\w\}$ is linearly independent.
Similarly to the above we can show
that $\{\la x_{\xi_{2n}}, x_{\xi_{2n+1}}\ra:n\in\w\}$ is
closed and discrete in $G^2$. On the other hand, $\la\la
x_{\xi_{2n}}, x_{2n+1}\ra:n\in\w\ra$ converges to $\la
x_\xi,x_\xi\ra$, a contradiction.


\begin{thebibliography}{ChGP??}

\bibitem{BanRav}
T. Banakh, A. Ravsky,
{\it Feebly compact paratopological groups}, preprint, arXiv:1003.5343.



\bibitem{BarRav}
S. Bardyla, A. Ravsky
{\it Closed subsets of compact-like topological spaces}, preprint,
arXiv:1907.12129.

\bibitem{ComRos}
W. W. Comfort, Kennet A. Ross,
{\it Pseudocompactness an uniform continuity in topological groups},
Pacif. J. Math. {\bf 16}:3 (1966), 483--496.

\bibitem{vDa80}
E.K. van Douwen,
{\it The product of two countably compact topological groups},
Trans. Amer. Math. Soc. \textbf{262}:2 (1980), 417--427.

\bibitem{DPSW}
A. Dow, J. R. Porter, R. M. Stephenson, Jr., and R. G. Woods,
{\it Spaces whose pseudocompact subspaces are closed subsets},
Appl. Gen. Topol. \textbf{5} (2004), 243--264.



\bibitem{G-FTW}
S. Garc\'ia-Ferreira, A.H. Tomita, S. Watson,
{\it Countably compact groups from a selective ultrafilter},
Proc. Amer. Math. Soc. {\bf 133}:3 (2005), 937Ц-943.

\bibitem{GutRav}
O. Gutik, A. Ravsky,
{\it On old and new classes of feebly compact spaces},
Visnyk of the Lviv Univ. Series Mech. Math. {\bf 85} (2018), 48--59.

\bibitem{HarvMil}
K. P. Hart, J. van Mill,
{\it A countably compact $ {H} $ such that $ {H} \times {H} $ is not countably compact},
Trans. Amer. Math. Soc. {\bf 323}:2, 811--821.

\bibitem{HvMR-GS}
M. Hru\v s\'ak, J. van Mill, U.A. Ramos-Garc\'ia, S. Shelah, {\it Countably compact groups without
non-trivial convergent sequences}, (presentable draft),
\url{https://matmor.unam.mx/~michael/preprints_files/Countably_compact.pdf}


\bibitem{MT}
R. Madariaga-Garcia, A. Tomita,
{\it Countably compact topological group topologies on free Abelian groups from selective ultrafilters},
Toology Appl. {\bf 154} (2007), 1470--1480.



\bibitem{Nov}
J. Nov\'ak,
{\it On the cartesian  product of two compact  spaces},
Fund. Math., {\bf 40} (1953), 106--112.

\bibitem{ScaSto}
C. T. Scarborough, A. H. Stone A.
{\it Products of nearly compact spaces}, 
Trans. Amer. Math. Soc. {\bf 124}:1 (1966), 131--147.

\bibitem{Ste}
R. M. Stephenson, Jr,
{\it Initially $\kappa$-compact and related compact spaces},
in K. Kunen, J. E. Vaughan (eds.), Handbook of Set-Theoretic Topology, Elsevier, 1984, 603--632.

\bibitem{ST}
P. Szeptycki, A. Tomita,
{\it HFD groups in the Solovay model},
Topology Appl. {\bf 156} (2009), 1807Ц-1810.

\bibitem{Ter}
H. Teresaka,
{\it On Cartesian  product of compact  spaces},
Osaka J. Math., {\bf 4} (1952), 11--15.

\bibitem{Tka2}
Mikhail Tkachenko,
{\it Generalization of the theorem by Comfort and Ross},
Ukr. Math. Jour., {\bf 41}:3 (1989), 377--382, in Russian.

\bibitem{Tka3}
Mikhail Tkachenko,
{\it Generalization of the theorem by Comfort and Ross II},
Ukr. Math. Jour., {\bf 41}:7 (1989), 939--944, in Russian.

\bibitem{Tka6}
Mikhail Tkachenko,
{\it Productive properties in topological groups},
preprint, (version April 18, 2013).
\url{http://ssdnm.mimuw.edu.pl/pliki/wyklady/Tkachenko-skrypt.pdf}

\bibitem{Vau}
J.E. Vaughan,
{\it Countably compact and sequentially compact spaces},
in K. Kunen, J. E. Vaughan (eds.), Handbook of Set-Theoretic Topology, Elsevier, 1984, 569--602.
\end{thebibliography}
\end{document}